\documentclass[11pt,A4paper]{article}

\usepackage[left=30mm,right=30mm,top=25mm,bottom=30mm]{geometry}
\usepackage{labelfig}
\usepackage{epsfig}
\usepackage{epstopdf}
\usepackage{soul}
\usepackage{xfrac}

\usepackage[percent]{overpic}

\usepackage{color}
\usepackage{amsthm,amsmath,amssymb}
\usepackage{booktabs}
\usepackage{mathpazo}
\usepackage{microtype}
\usepackage{overpic}
\usepackage{bm}
\usepackage{sectsty}
\usepackage[
	pdftitle={PDFTitle},
	pdfauthor={Hugo Parlier},
	ocgcolorlinks,
	linkcolor=linkred,
	citecolor=linkred,
	urlcolor=linkblue]
{hyperref}

\definecolor{linkred}{RGB}{227, 115, 131} %Indigo
\definecolor{linkblue}{RGB}{16, 78, 139}

\usepackage[hang,flushmargin]{footmisc}
\usepackage{enumitem}
\usepackage{titlesec}
	\titlespacing{\section}{0pt}{12pt}{0pt}
	\titlespacing{\subsection}{0pt}{6pt}{0pt}
	
\titlelabel{\thetitle.\quad}

\makeatletter 

\long\def\@footnotetext#1{% 
\H@@footnotetext{% 
\ifHy@nesting 
\hyper@@anchor{\@currentHref}{#1}% 
\else 
\Hy@raisedlink{\hyper@@anchor{\@currentHref}{\relax}}#1% 
\fi 
}}

\def\@footnotemark{% 
\leavevmode 
\ifhmode\edef\@x@sf{\the\spacefactor}\nobreak\fi 
\H@refstepcounter{Hfootnote}% 
\hyper@makecurrent{Hfootnote}% 
\hyper@linkstart{link}{\@currentHref}% 
\@makefnmark 
\hyper@linkend 
\ifhmode\spacefactor\@x@sf\fi 
\relax 
}% 

\ifFN@multiplefootnote% 
\renewcommand*\@footnotemark{% 
\leavevmode 
\ifhmode 
\edef\@x@sf{\the\spacefactor}% 
\FN@mf@check 
\nobreak 
\fi 
\H@refstepcounter{Hfootnote}% 
\hyper@makecurrent{Hfootnote}% 
\hyper@linkstart{link}{\@currentHref}% 
\@makefnmark 
\hyper@linkend 
\ifFN@pp@towrite 
\FN@pp@writetemp 
\FN@pp@towritefalse 
\fi 
\FN@mf@prepare 
\ifhmode\spacefactor\@x@sf\fi 
\relax% 
}% 
\fi 

\makeatother 

\theoremstyle{plain}
\newtheorem{theorem}{Theorem}[section]

\newtheorem{lemma}[theorem]{Lemma}

\theoremstyle{definition}

\newcommand{\area}{{\rm area}}
\newcommand{\arcsinh}{{\,\rm arcsinh}}

\sectionfont{\large \bfseries}
\subsectionfont{\normalsize}

\long\def\symbolfootnote[#1]#2{\begingroup%
\def\thefootnote{\fnsymbol{footnote}}\footnote[#1]{#2}\endgroup}

\def\blfootnote{\xdef\@thefnmark{}\@footnotetext}

\usepackage{mathtools}

\usepackage[sort,nocompress]{cite}

\begin{document}

{\Large \bfseries A shorter note on shorter pants}

{\large 
Hugo Parlier\symbolfootnote[1]{\small Supported by the Luxembourg National Research Fund OPEN grant O19/13865598.
}
}

\vspace{0.3cm}
{\bf Abstract.}
This note is about variations on a theorem of Bers about short pants decompositions of surfaces. It contains a version for surfaces with boundary but also a slight improvement on the best known bound for closed surfaces.
\vspace{0.3cm}

\section{Introduction} \label{sec:intro}

A theorem of Bers asserts that any closed hyperbolic surface admits a short pants decomposition. More precisely, Bers exhibited the existence of a constant, that only depends on the topology of the surface, which bounds the length of the shortest pants decomposition of {\it any} hyperbolic surface with the given topology \cite{Bers1, Bers}. Quite a bit of effort has gone into quantifying these constants in terms of topology, including in more general cases such as surfaces with cusps or boundary and for Riemaniann surfaces \cite{Balacheff-Parlier,BPS,BuserBook,Buser-Seppala,Gendulphe,GPY}. 

The main goal of this note is to show the following.
\begin{theorem}\label{thm:intro}
Let $X$ be a hyperbolic surface, possibly with geodesic boundary, and of finite area. Then $X$ admits a pants decomposition where each curve is of length at most 
$$
\max\{ \ell(\partial X), \area(X)\}.
$$
\end{theorem}
While the context is slightly different (here we allow boundary), the techniques are very close to \cite{ParlierShortPants}. The main novelty is that the proof has been simplified to its bare essentials, and in the case where the surface is closed, the above statement is a slight improvement on the best known bounds.

\section{Setup}
Our surfaces will be orientable, finite-type and hyperbolic. They will be either closed or with boundary geodesics, where we use the convention that a cusp is a boundary geodesic of length $0$. If $X$ is a hyperbolic surface, and $\gamma$ is a non-trivial homotopy class of closed curve on $X$, the quantity $\ell_X(\gamma)$ is the length of the unique shortest closed geodesic freely homotopic to $\gamma$. If the surface $X$ is implicit, this will just be written as $\ell(\gamma)$. 

We're interested in pants decompositions of surfaces, that is maximal collections of disjoint simple closed geodesics. They decompose a surface into pants (topologically three-holed spheres). The length of a pants decomposition is, by convention, the maximal length of the curves in the decomposition. A hyperbolic pair of pants has area $2\pi$, so a hyperbolic surface $X$ has area $\area(X)$ equal to $2 \pi$ times the number of pairs of pants needed to build it.

This note relies on two tools. The first is well-known and concerns surfaces with non-empty boundary. 
\begin{lemma}[Length expansion lemma]\label{lem:LE} Let $X$ be a finite-type hyperbolic surface with boundary geodesics of length $(\ell_1,\hdots,\ell_n)$ and let $\varepsilon>0$. Then there exists a hyperbolic surface $X' \cong X$ with boundary geodesics of length $(\ell_1+\varepsilon,\hdots,\ell_n)$ and such that any non-trivial simple closed curve $\gamma\subset \Sigma$ satisfies
$$
\ell_{X'}(\gamma)>\ell_{X}(\gamma).
$$
\end{lemma}

This result, claimed in \cite{ThurstonSpine}, continues to be used in various forms (see for example \cite{Danciger-Gueritaud-Kassel}, \cite{ParlierLengths} for a direct proof and \cite{Papadopoulos-Theret} for a stronger version).

The second tool is already used in \cite{ParlierShortPants} and we state it here in the form of a lemma about pants. We provide a short proof idea for convenience.

\begin{lemma}\label{lem:pants} Let $Y$ be a hyperbolic pair of pants with geodesic boundary curves $\alpha, \beta, \gamma$ and geodesic seams $c$ and $h$ as in Figure \ref{fig:pants}. If $\ell(c) \leq 2 \arcsinh(1)$ then $\ell(\alpha) + \ell(\beta) > \ell(\gamma)$. If however $\ell(h) \leq 2 \arcsinh(1)$ then $\ell(\alpha) + \ell(\beta) < \ell(\gamma)$. 
\end{lemma}
\begin{figure}[h]
%\ShowGrid
\leavevmode \SetLabels
\L(.54*.29) $c$\\
\L(.453*.72) $h$\\
\endSetLabels
\begin{center}
\AffixLabels{\centerline{\epsfig{file=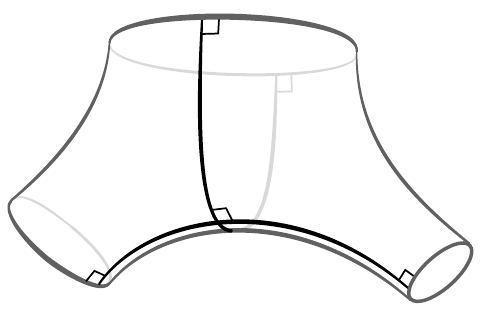,width=5.0cm,angle=0}}}
\vspace{-18pt}
\end{center}
\caption{The paths $c$ and $h$} \label{fig:pants}
\end{figure}

\begin{proof}
Both statements follow from standard trigonometry computations. We first consider the hexagon with non-adjacent sides of length $\ell(\alpha)/2, \ell(\beta)/2 $ and $\ell(\gamma)/2$, two copies of which form the pair of pants. In the first case, we can use $\ell(c) \leq 2 \arcsinh(1)$ and so $\cosh(\ell(c))\leq 3$ to obtain
\begin{eqnarray*}
\cosh \frac{\ell(\gamma)}{2}&=& \sinh \frac{\ell(\alpha)}{2}\sinh \frac{\ell(\beta)}{2} \cosh \ell(c) - \cosh \frac{\ell(\alpha)}{2}\cosh \frac{\ell(\beta)}{2}\\
&\leq &  \sinh \frac{\ell(\alpha)}{2}\sinh \frac{\ell(\beta)}{2} 3 - \cosh \frac{\ell(\alpha)}{2}\cosh \frac{\ell(\beta)}{2}\\
&<& \cosh\left(  \frac{\ell(\alpha)}{2}+ \frac{\ell(\beta)}{2} \right).
\end{eqnarray*}
The second statement follows from splitting the hexagon into two pentagons along the perpendicular path of length $\ell(h)/2$ between $c$ and $\gamma$ (see Figure \ref{fig:pent}).
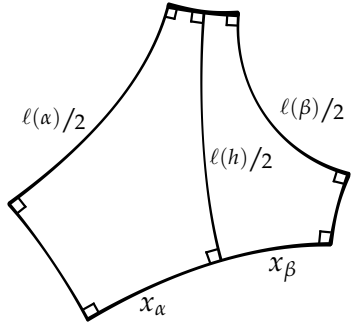
\begin{figure}[h]
%\ShowGrid
\leavevmode \SetLabels
\L(.52*.48) $\sfrac{\ell(h)}{2}$\\
\L(.36*.6) $\sfrac{\ell(\alpha)}{2}$\\
\L(.58*.63) $\sfrac{\ell(\beta)}{2}$\\
\L(.46*.05) $x_\alpha$\\
\L(.57*.17) $x_\beta$\\
\endSetLabels
\begin{center}
\AffixLabels{\centerline{\epsfig{file=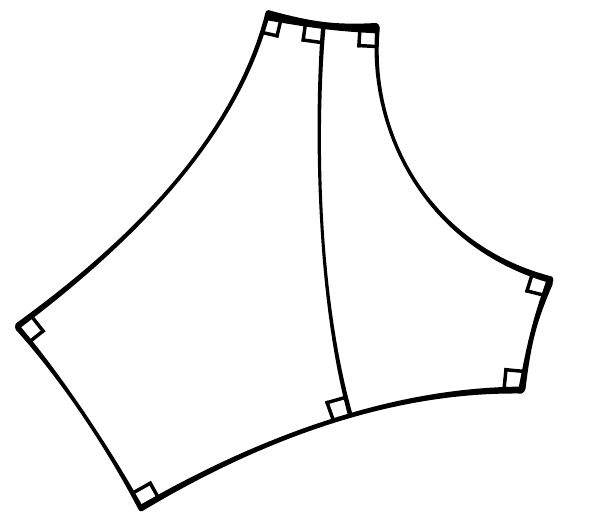,width=5.0cm,angle=0}}}
\vspace{-18pt}
\end{center}
\caption{The hexagon and pentagons in the second case}\label{fig:pent}
\end{figure}
Let $x_\alpha$ and $x_\beta$ be the lengths of the two subpaths of $\gamma$ as indicated in the figure. Now using $h \leq 2 \arcsinh(1)$ we have
\begin{eqnarray*}
\sinh x_\alpha &\leq  \sinh x_\alpha \sinh \frac{h}{2}  &= \cosh \frac{\ell(\alpha)}{2} \\
\sinh x_\beta &\leq  \sinh x_\beta \sinh \frac{h}{2}  &= \cosh \frac{\ell(\beta)}{2} 
\end{eqnarray*}
and thus
$$\frac{\ell(\gamma)}{2}= x_\alpha+x_\beta < \frac{\ell(\alpha) + \ell(\beta)}{2}$$
as claimed.
\end{proof}

\section{Proof of Theorem \ref{thm:intro}}
We begin with the case when $X$ is not closed. We argue by induction on the number of pairs of pants needed to construct $X$. The initial step of the induction, when $X$ is a pair of pants, holds by definition. 

Now if $\ell(\partial X) < \area(X)$, then by Lemma \ref{lem:LE} we can increase the boundary length while increasing the length of all simple closed geodesics until the length is equal to $\area(X)$. As we are proving an upper bound on curve lengths, if the statement holds for the resulting surface, it will also hold for the initial surface. Hence we can suppose that $\ell(\partial X) \geq \area(X)$. 

For $r>0$, consider an $r$-neighborhood of $\partial X$. Provided $r$ is small enough, this neighborhood is embedded and has area $\sinh(r) \ell(\partial X) < \area(X)$. If we set $r_0$ to be the supremum of all values of $r$ where the neighborhood is embedded we have
$$
r_0 < \arcsinh\left( \frac{\area(X)}{\ell(\partial X)} \right) \leq \arcsinh(1).
$$
(Note the strict inequality still holds because the closed neighborhood cannot entirely cover the surface.) In this limit case, we have a non-trivial geodesic arc of length $2 r_0 \leq 2 \arcsinh(1)$ either between distinct boundary curvesor from one boundary curve to itself (see the left and right illustrations in Figure \ref{fig:topo}). 

\begin{figure}[h]
%\ShowGrid
\leavevmode \SetLabels
\L(.32*.25) $2r_0$\\
\L(.62*.65) $2r_0$\\
\endSetLabels
\begin{center}
\AffixLabels{\centerline{\epsfig{file=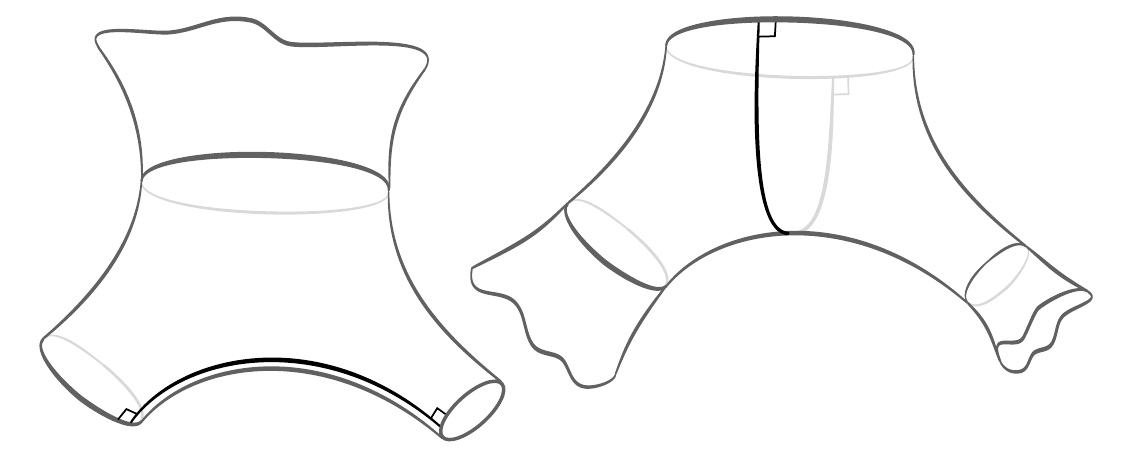,width=10.0cm,angle=0}}}
\vspace{-18pt}
\end{center}
\caption{The two topological types for the path} \label{fig:topo}
\end{figure}

In both cases, associated to this arc we have an embedded pair of pants which has either one or two curves belonging to $\partial X$. We can apply Lemma \ref{lem:pants} to this pair of pants, and remove it from $X$, to obtain a surface $X'$ with one less pair of pants and $\ell(\partial X') < \ell(\partial X)$. By induction, we are done.

We now have to prove the result when $X$ is closed. We will cut $X$ along a shortest simple closed geodesic (a systole) and then refer to the case of a surface with boundary to complete the systole into a full pants decomposition. 
\begin{lemma}\label{lem:systole}
Any closed hyperbolic surface $X$ has a systole of length strictly less than $\area(X)/2$.
\end{lemma}
\begin{proof}
Let $\alpha$ be a systole and $s$ its length. By a standard cut and paste argument, the $\frac{s}{4}$ neighborhood of $\alpha$ is embedded (otherwise it is easy to construct a non-trivial curve of shorter length). The area of this neighborhood is 
$$
2\, s \sinh \frac{s}{4} < \area(X).
$$
Now if $s \geq 4 \arcsinh(1)$, the result holds. If not, $s  < 4 \arcsinh(1) < 4 < 2\pi$. And any closed surface $X$ has area at least $4\pi$ and so the result follows.
\end{proof}
The main result then follows by cutting along the systole to obtain a surface with boundary of length strictly less than $\area(X)$ (which may be possibly disconnected, but that only makes the result easier).

{\it Address:}\\
DMATH, FSTM, University of Luxembourg, Esch-sur-Alzette, Luxembourg

{\it Email:}\\
hugo.parlier@uni.lu

\end{document}